\documentclass[a4paper]{article}
\usepackage[latin1]{inputenc}
\usepackage{graphicx}
\usepackage{amssymb, amsmath, amsfonts, amsthm}
\pdfoutput=1
\usepackage{subfig}

\def\P{\mathcal{P}}
\def\I{\mathcal{I}}

\newcommand{\dgavf}{\dg_{\mathrm{AVF}}}
\newcommand{\dgci}{\dg_{\mathrm{CI}}}
\newcommand{\dgsci}{\dg_{\mathrm{SCI}}}
\newcommand{\dg}{\overline{\nabla}}
\newcommand{\rank}{\ensuremath{\mathrm{rank}}}
\newcommand{\Tb}{\ensuremath{T}}

\newtheorem{theorem}{Theorem}[section]
\newtheorem{lemma}[theorem]{Lemma}
\newtheorem{definition}[theorem]{Definition}
\newtheorem{example}[theorem]{Example}
\numberwithin{equation}{section}
\numberwithin{table}{section}
\numberwithin{figure}{section}

\title{Preserving multiple first integrals by \\discrete gradients}
\author{Morten Dahlby \and Brynjulf Owren \and Takaharu Yaguchi}

\begin{document}
\maketitle

\begin{abstract}
 We consider systems of ordinary differential equations with known first integrals.  The notion of a discrete tangent space is introduced as the orthogonal complement of an arbitrary set of discrete gradients. Integrators which exactly conserve all the first integrals simultaneously are then defined. 
 In both cases we start from an arbitrary method of a prescribed order (say, a Runge-Kutta scheme) and modify it using two approaches: one based on projection and one based one local coordinates.
 The methods are tested on the Kepler problem. 
 \end{abstract}
 
 \section{Introduction}
 A system of ordinary differential equations which preserves a first integral $H(y)$ can be written in the form
 \begin{equation} \label{eq:ipode}
 \dot{y}=f(y) = S(y)\nabla H(y),\ y\in\mathbb{R}^m,
 \end{equation}
where $S(y)$ is an antisymmetric matrix, see \cite{mclachlan99giu}.
 An approximate numerical solution, $y^n\approx y(t^n),\ n\geq 1$, is said to be integral preserving if
 $H(y^n)=H(y^0),\ n\geq 1$.  Energy preserving methods go all the way back to the seminal paper by Courant, Friedrichs, Lewy \cite{courant28udp}, where an energy preserving difference scheme is derived and the property is used to prove convergence of the scheme. Recently, it has become increasingly popular to study conservative schemes as discrete dynamical systems in their own right, attempting to mimic properties of the continuous system by introducing suitable discrete counterparts. An example of importance in this note is the replacement of the   gradient in \eqref{eq:ipode} by a discrete gradient operator.
 
 The discrete gradient method is one of the most prevalent approaches in the literature. It was first systematically treated by Gonzalez \cite{gonzalez96tia} and McLachlan et al.\ \cite{mclachlan99giu}.
 The idea is to introduce a discrete approximation to the gradient, letting $\dg H:\mathbb{R}^m\times\mathbb{R}^m\rightarrow\mathbb{R}^m$ be a continuous map satisfying
 \begin{align*}
     H(u)-H(v)&=\dg H(v,u)^{\top} (u-v), \\
     \dg H(u,u) &= \nabla H(u).
 \end{align*}
 The existence of such discrete gradients is well established in the literature, see for instance the monograph by Hairer et al.\ \cite{hairer06gni} or the papers \cite{gonzalez96tia} and \cite{mclachlan99giu}. Their construction is not unique,  we give here two different examples.
 The Averaged Vector Field (AVF) gradient \cite{MR2451073} is defined as
 \begin{equation} \label{eq:avf}
    \dgavf H (v,u) = \int_0^1 \nabla H(\xi u + (1-\xi)v)\,\mathrm{d}\xi.
 \end{equation}
 The coordinate increment method \cite{itoh88hcd} is defined in terms of the coordinates of the vectors $v$ and $u$, the $i$th component of 
 $\dg H(v,u)$ is then given as
 \begin{equation} \label{eq:itohabe}
 (\dgci H(v,u))_i = \frac{H(u_1,\ldots,u_{i},v_{i+1},\ldots,v_m)-H(u_1,\ldots,u_{i-1},v_{i},\ldots,v_m)}
 {u_{i}-v_{i}}.
 \end{equation}
 An important difference between these two discrete gradients is that \eqref{eq:avf} is symmetric, 
 $\dgavf H(v,u)=\dgavf H(u,v)$, while \eqref{eq:itohabe} is not. However, note that a symmetric version of the coordinate increment discrete gradient can be constructed by 
 \begin{equation} \label{eq:symitohabe}
 	\dgsci H(v,u)=\frac12\left(\dgci H(v,u)+\dgci H(u,v)\right).
 \end{equation}
 Once a discrete gradient has been found, one immediately obtains an integral preserving method by simply letting
 \[
     \frac{y^{n+1}-y^n}{h} = \overline{S}(y^n,y^{n+1})\,\dg H(y^n, y^{n+1}).
 \]
 Here $h$ is the time step, and $\overline{S}(y^n,y^{n+1})$ is some skew-symmetric approximation to the matrix $S$, one would normally require that $S(y)=\overline{S}(y,y)$. We remark that discrete gradient methods are implicit.
 
 In this note we consider the case where there are more than one first integral and the objective is to preserve any number of such invariants simultaneously. Some earlier attempts to achieve this include the papers \cite{mclachlan99giu, quispel99son} in which the antisymmetric matrix $S(y)$ is replaced by an antisymmetric tensor taking discrete gradients of all integrals to be preserved as input. A formula for this antisymmetric tensor is given. Another approach is an integrator for a class of separable Hamiltonian systems (\cite{minesaki06nni} and references therein), where the integrator which preserves all integrals is designed based on separation of variables by the Kustaanheimo--Stiefel transformation \cite{stiefel71lar}. This transformation was also adopted in the development of the ``exact'' integrator for the Kepler problem by Kozlov (\cite{kozlov07cdo}, see also \cite{cieslinski07}). 
It would be also noteworthy that Labudde and Greenspan \cite{labudde74dma} proposed an energy-and-angular-momentum-preserving integrator for the differential equations of motion of classical mechanics, and they developed similar integrators of high order of convergence in the sequels \cite{labudde76eai, labudde76eat}.
Energy and linear momentum are exactly conserved, but for the system case \cite{labudde76eat} angular momentum is not. Another approach is used by Simo et al.\ in \cite{MR1187632} to develop schemes that preserve energy and momentum. 

We shall instead present an approach which does not rely on finding such a tensor nor a structure of the equation, we only assume knowledge of the first integrals to be preserved as well as the ODE vector field itself. Examples of such invariants are the energy and momentum, but our approach is not limited to these. The discrete gradients are essential to the algorithm we develop, but note that the methods themselves are not discrete gradient methods in the usual sense. 
After defining the general method, we present two particular cases based on projection and local coordinates, respectively. Both use an underlying scheme of arbitrary order $p$, we prove that the resulting schemes retain this order, regardless of the choice of discrete gradient. 

In Section \ref{se:numex} we apply  the new methods to the Kepler problem, a system with four degrees of freedom and three independent first integrals. 
We illustrate our approach by preserving combinations of one or more of these three integrals. An interesting question is whether the present methods based on discrete gradients perform better than the more standard projection methods which make use of the exact gradients of the first integrals.
 We show two examples where the different approaches are compared.
 In the examples we use Runge-Kutta methods of different order as the underlying schemes.
The Kepler problem is of course a well-known and popular test case, and methods which preserve one or more integrals for this particular problem can be found in e.g.   \cite{cieslinski07, cieslinski10, kozlov07cdo, minesaki04anc}, the methods in these references are derived by means of the Kustaanheimo--Stiefel transformation.

 \section{Preserving multiple invariants}
 
 Suppose that an ODE system \eqref{eq:ipode} possesses  $q\geq 1$ independent first integrals, $H_1(y),\ldots,H_q(y)$. These invariants foliate $\mathbb{R}^m$ into $(m-q)$-dimensional submanifolds (leaves)
 \[
    M=M_c = \{y\in\mathbb{R}^m : H_1(y)=c_1, H_2(y)=c_2,\ldots,H_q(y)=c_q\}.
 \]
 The tangent space $T_yM$ of $M$ at $y$ is the orthogonal complement to
 \[
 	\text{span}\{\nabla H_1(y),\ldots,\nabla H_q(y)\}. 
\]	
For simplicity we write only $M$ for $M_c$ for the rest of this paper.
 
 \begin{definition}
Let $\dg$ be a fixed discrete gradient operator and let $H_1,\ldots,H_q$ be independent first integrals.
The discrete tangent  space at  $(u,v)\in \mathbb{R}^m\times\mathbb{R}^m$ is
\[
    T_{(v,u)}M = \{\eta\in\mathbb{R}^m : \langle\dg H_j(v,u),\eta\rangle=0,\ 1\leq j\leq q\}.
\]
A vector $\eta=\eta_{(v,u)}\in T_{(v,u)}M$ is called a discrete tangent vector.
 \end{definition}
 \noindent Note that this definition causes $T_{(y,y)}M=T_yM$. 
 \begin{lemma} \label{lem:cond}
 Any integrator satisfying
 \[
      y^{n+1}-y^n = \eta_{(y^n,y^{n+1})} \in T_{(y^{n},y^{n+1})}M
 \]
 preserves all integrals, in the sense that $H_i(y^{n+1})=H_i(y^n),\ 1\leq i\leq q$.
 \end{lemma}
 \begin{proof} 
 For any $i$ we compute
\begin{multline*}
   H_i(y^{n+1})-H_i(y^n) = \dg H_i(y^n,y^{n+1})^{\top}(y^{n+1}-y^n) \\= \dg H_i(y^n,y^{n+1} )^{\top}\eta_{(y^n,y^{n+1})}=0.
 \end{multline*}
 \end{proof}
 This way of devising integral preserving schemes is similar, but slightly different to the presentation in e.g.\ \cite{hairer06gni}, \cite{mclachlan99giu} and \cite{MR2451073}.
 In this paper we will outline two ways of ensuring that the condition of Lemma~\ref{lem:cond} is satisfied -- projection and local coordinates. We emphasise, however, that there are other ways of satsifying Lemma~\ref{lem:cond}. One example taken from \cite{mclachlan99giu} is
 \begin{equation*}
 \eta_{(y^n,y^{n+1})}=h\overline{S}(y^n,y^{n+1})\overline{\nabla}H_1(y^n,y^{n+1})\dots\overline{\nabla}H_q(y^n,y^{n+1}),
 \end{equation*}
where $\overline{S}$ is a $q$-dimensional skew-symmetric tensor. 
 
 \subsection{Projection}\label{se:proj}
 We consider \eqref{eq:ipode}  
with the first integrals
$H_1(y),\dots,H_q(y)$.
We propose the projection scheme 
\begin{align}\label{scm:A}
& u^{n+1} =  \phi_h(y^{n}), \quad
 y^{n+1} = y^{n} + \P(y^{n},y^{n+1}) (u^{n+1} -y^{n})
\end{align}
where $\phi_h$ is the discrete flow that defines an arbitrary method of order $p$
\begin{align*}
 y(t + h) - u^{n+1} =   y(t+h) - \phi_h(y(t))  = \mathcal{O}(h^{p+1}), 
\end{align*}
and $\P(y^{n}, y^{n+1})$ is a smooth projection operator onto the discrete tangent space $T_{(y^{n}, y^{n+1})} M$.  
An alternative method is 
\begin{align}\label{scm:B}
& y^{n+1} = y^{n} + h \P(y^{n}, y^{n+1}) \psi_h(y^{n}, y^{n+1})
\end{align}
where $\psi_h$ is the increment function that defines a method of the form
\begin{align*}
 y^{n+1} =  y^{n} + h\psi_h(y^{n}, y^{n+1}).
\end{align*}
This method is itself assumed to be of order $p$, that is
\begin{align}\label{eqn:b-order-p}
  y(t+h) -  y(t) - h\psi_h(y(t), y(t+h)) = \mathcal{O}(h^{p+1}). 
\end{align}
We remark in passing that if both the discrete gradients $\overline{\nabla}H_i$ and the increment function $\psi_h$ are symmetric, then the method \eqref{scm:B} is symmetric. 

\begin{example}
Using Runge-Kutta as the underlying scheme $\phi_h$ we can construct examples of \eqref{scm:A}. The unprojected solution is given as
\begin{equation*}
	u^{n+1}=y^n+h\sum_{i=1}^sb_{i}k_i,
\end{equation*}
where $k_1,\dots,k_s$ are the solutions to the (possibly implicit) equations
\begin{equation*}
	k_i=f\left(y^n+h\sum_{j=1}^sa_{ij}k_j\right).
\end{equation*}
Even if the Runge-Kutta scheme is explicit the scheme \eqref{scm:A} will be implicit since $y^{n+1}$ appears in the projection operator
\begin{equation*}
 	y^{n+1} = y^{n} + h\P(y^{n},y^{n+1}) \left(\sum_{i=1}^sb_ik_i\right).
\end{equation*}

The difference between \eqref{scm:A} and \eqref{scm:B} is subtle, but to illustrate that they are in fact distinct we consider the implicit midpoint method as the underlying scheme and we get for the two methods
\begin{align*}
 	y^{n+1} &= y^{n} + h\P(y^{n},y^{n+1}) f\left(\frac{y^n+u^{n+1}}2\right),\\
 	y^{n+1} &= y^{n} + h\P(y^{n},y^{n+1}) f\left(\frac{y^n+y^{n+1}}2\right), 
\end{align*}
where in the former method $u^{n+1}$ is computed by solving
\begin{align*}
 u^{n+1} = y^{n} + h f\left(\frac{y^n+u^{n+1}}2\right). 
\end{align*}
\end{example}

\begin{theorem}\label{thm:A}
The schemes \eqref{scm:A} and \eqref{scm:B} are of order $p$, that is 
\begin{align}
\begin{split}
 y(t+h) - y(t) - \P(y(t), y(t+h)) (u^{n+1} -y(t)) &= \mathcal{O}(h^{p+1}),\\ u^{n+1} &= \phi_h(y(t)),
 \end{split}
 \label{eq:ordscmA}
 \end{align}
 and
 \begin{equation}
  y(t+h) - y(t) - h \P(y(t), y(t+h)) \psi_h(y(t), y(t+h)) = \mathcal{O}(h^{p+1}).\label{eq:ordscmB}
\end{equation}
\end{theorem}

\begin{proof}
We use the shorthand notation $\P$ for $\P(y(t),y(t+h))$ in this proof. To prove \eqref{eq:ordscmA}, we compute
\begin{align*}
& y(t+h) - y(t) - \P (u^{n+1} -y(t)) \\
& = y(t+h) - y(t) - (\I - (\I - \P)) (u^{n+1} -y(t)) \\
& = y(t+h) - y(t) -  (u^{n+1} -y(t)) + (\I - \P) (u^{n+1} -y(t)). 
 \end{align*}
Since $\phi_h$ is of order $p$, we have
\begin{align}\label{eqn:a-order-p}
 y(t+h) - y(t)
 - (u^{n+1} - y(t)) 
=  y(t+h) - u^{n+1} =
\mathcal{O}(h^{p+1}). 
\end{align}
Therefore if we have
\begin{align*}
(\I - \P) (u^{n+1} -y(t)) = \mathcal{O}(h^{p+1}), 
\end{align*}
the proof is completed. This estimate is obtained in the following way. 
Because the image of $\I - \P(y(t), y(t+h))$ is spanned by 
\begin{equation*}
	\left\{\dg H_1 (y(t), y(t+h)),\dots,\dg H_q (y(t), y(t+h))\right\},
\end{equation*} 
it is enough to show 
\begin{align*}
\dg H_i (y(t), y(t+h)) \cdot (u^{n+1} -y(t)) = \mathcal{O}(h^{p+1}). 
\end{align*}
From \eqref{eqn:a-order-p} we obtain
\begin{align*}
  u^{n+1} - y(t)
 = y(t+h) - y(t)
 + \mathcal{O}(h^{p+1}), 
\end{align*}
and hence
\begin{align*}
& \dg H_i (y(t), y(t+h)) \cdot (u^{n+1} -y(t))  \\
& = \dg H_i (y(t), y(t+h)) \cdot (y(t+h) - y(t) + \mathcal{O}(h^{p+1})) \\
& = H_i (y(t+h)) - H_i(y(t)) + \mathcal{O}(h^{p+1}) \\
& =  \mathcal{O}(h^{p+1}). 
\end{align*}
The last equality is from the conservation property of the original equation. 
The proof of \eqref{eq:ordscmB} is almost identical and therefore omitted.
\end{proof}

\noindent We remark that in what we call standard projection methods $u^{n+1}$ is projected orthogonally onto the manifold $M$ by computing
\begin{equation*}
\min\| y^{n+1} - u^{n+1}\|  \quad \mbox{subject to } y^{n+1} \in {M}.
\end{equation*}
This can be achieved for instance by using Lagrange multipliers (see \cite{hairer06gni}, for instance). This approach differs from ours since we project along the discrete gradients which depend on the end point $y^{n+1}$.

\paragraph{Computing the projector.} A simple and straightforward way of obtaining the projector $\mathcal{P}(y^n,y^{n+1})$ is as follows:
Define the $(q\times m)$-matrix $Y=Y(y^n,y^{n+1})$ whose columns are the discrete gradients $\dg H_i(y^n,y^{n+1}),\ i=1,\ldots,q$. Compute a reduced $QR$-decomposition $QR=Y$ where $Q\in\mathbb{R}^{m\times q}$ and $R\in\mathbb{R}^{q\times q}$. Then define the projection matrix as $\mathcal{P}(y^n,y^{n+1})=I-QQ^{\top}$.

 \subsection{Local coordinates} 
The local coordinates approach presented here is basically of the same type as the
standard method by Potra and Rheinboldt (\cite{MR1142054}, see also \cite{hairer06gni, MR1142057}). One important difference is that
our local coordinates are algorithmically constructed by using discrete gradients.
We also present an "automatic differentiation" algorithm of the coordinate map, which could be used to increase the efficiency of the computations.  

 Inspired by \cite{celledoni02aco}, 
 we consider local coordinates on a chart containing
$y^0$ by defining a map $\eta\mapsto y=\chi(\eta)$. The map is defined implicitly by 
\begin{equation} \label{eq:imprep}
\chi(\eta)=y:\ y-y^0=\Tb(y^0,y)\eta,
\end{equation}
where $\Tb(y^0,y)$ is a smooth $m\times (m-q)$-matrix whose columns form a basis for the left nullspace (orthogonal column complement) of
the matrix $Y(y^0,y)=[\dg H_1(y^0,y),\dots,\dg H_q(y^0,y)]$. We suppress the dependency on $y^0$ and use the shorthand notation $T(y)$ and $Y(y)$ for the rest of this paper. 
\begin{lemma}\label{lem:atlas}
Suppose that ${\nabla}H_1(y^0), \dots, {\nabla}H_q(y^0)$ are linearly independent for all $y^0 \in M$.
Suppose also that for all $y^0 \in M$, $\dg H_1(y^0,y), \dots, \dg H_q(y^0,y)$ are 
$\mathrm{C}^\infty$ with respect to $y$.
Then the following statements hold.
\begin{enumerate}
\item $\eqref{eq:imprep}$ defines a one-to-one map $\omega_{y^0}$ in a neighborhood ${N}_{y^0} \subset M$.
$\omega_{y^0}$ and $\omega^{-1}_{y^0}$ are $\mathrm{C}^\infty$. 
\item The collection of the pairs $\{(N_{y^0}, \omega_{y^0}) \mid y^0 \in M\}$ forms an atlas of $M$. 
\end{enumerate}
\end{lemma}
\begin{proof}
\begin{enumerate}
\item 
From the continuity of the discrete gradients, we deduce that
for all $y^0 \in M$, there exists a neighborhood $\tilde{N}_{y^0}$ in $\mathbb{R}^m$ of $y^0$ in which $\dg H_1(y^0,y), \ldots, \dg H_q(y^0,y)$ are linearly independent.  
For $y \in \tilde{N}_{y^0}$, $\Tb(y)$ admits the QR decomposition $\Tb(y) = QR$ and $\eta$ is obtained by $\eta = (R^\top R)^{-1} Q^\top(y-y^0) = (\Tb^\top \Tb)^{-1} \Tb^{\top}(y-y^0)$. This is a ${\mathrm{C}^\infty}$ function.
Conversely, the Jacobian matrix of the function $\eta(y)$ at $y=y^0$ is 
\begin{align*}
\frac{\partial \eta}{\partial y}(y^0) =  (\Tb^\top \Tb)^{-1} \Tb^{\top}
\end{align*}
and hence
\begin{align*}
\rank 
\frac{\partial \eta}{\partial y}(y^0) = \rank  (\Tb^\top \Tb)^{-1} \Tb^{\top} = m-q.
\end{align*}
Thus $\omega$ is defined in a neighborhood $\bar{N}_{y^0}$ of $y^0$ by the implicit function theorem and is $\mathrm{C}^\infty$.
The proof is completed by letting $N_{y^0} = \bar{N}_{y^0} \cap \tilde{N}_{y^0} \cap M$.
\item This is immediately obtained from the first statement.
\end{enumerate}
\end{proof}

Consider now the curve $\eta(t)$ and
let $y(t)=\chi(\eta(t))$. We differentiate the curve to obtain from \eqref{eq:imprep}
\[
     \dot{y}(t) = \Tb'_{y(t)}( \dot{y}(t))\eta(t) + \Tb(y(t))\dot{\eta}(t).
\]
From this we compute
\begin{equation}\label{eq:etaeq}
    \dot{\eta} = -\Tb^{\top}(\chi\circ\eta) \Tb_{\chi\circ\eta}'(f(\chi\circ\eta))\eta+\Tb^{\top}(\chi\circ\eta)f(\chi\circ\eta)
\end{equation}
where the original ODE is $\dot{y}=f(y)$. The method we propose takes one step as follows
\begin{enumerate}
\item Let $\eta_0=0$.
\item Take a step with any $p$th order method applied the ODE \eqref{eq:etaeq} using
$y^0=y^n$ in \eqref{eq:imprep}. The result is $\eta_1$.
\item Compute $y^{n+1}=\chi(\eta_1)$.
\end{enumerate}
We immediately obtain the next theorem from Lemma \ref{lem:atlas}, because the solution curve lies in $M$ and a $p$th order method is applied in a chart of $M$.
\begin{theorem}
Under the assumptions of Lemma \ref{lem:atlas}, the above scheme is of order $p$.
\end{theorem}

The main difficulty in this approach is the computation of the derivative map $\Tb_y'(\zeta)$ for arbitrary values of $y\in\mathbb{R}^m$ and $\zeta\in\mathbb{R}^m$.   This is needed explicitly in the integration algorithm, but may also be a useful tool in computing the coordinate map \eqref{eq:imprep}.
We may define $\Tb(y)$ as the last $m-q$ columns of the $m\times m$-matrix $Q(y)$ defined through a QR-decomposition where
$Y(y)=Q(y)R(y)$ and where we have used the shorthand notation
\begin{equation*}
   Y(y) = [\dg H_1(y^n,y),\ldots,\dg H_q(y^n,y)]
\end{equation*}

We realise the QR-decomposition by means of the Householder method, applying a sequence of $q$ elementary orthogonal transformations to the matrix $Y(y)$ as described in most elementary text books in numerical linear algebra, see e.g \cite{trefethen97nla}. Each transformation is of the form
\begin{equation*}
   Q_k = I - 2v_kv_k^\top,\qquad v_k\in\mathbb{R}^m,\ v_k^{\top} v_k=1,
\end{equation*}
and its aim is to eliminate all elements under the diagonal in the $k$th column of the matrix to which it is applied. 

In order to explain how we compute the derivative $Q'_y(\zeta) =: DQ(y,\zeta)$, we first review the Householder method.
\begin{tabbing}
xxx\=xxx\=xxxxxxxxxxxxxxxxxxxxxxxxxxxxxxxx\=xxx\=x \kill
$Y^{(1)} := Y$   \\
\textbf{for} $k=1:q$, \\
    \>$w_k= \Pi_k Y_k^{(k)}-\|\Pi_kY_k^{(k)}\|\mathrm{e}_k$  \\
     \> $v_k = \frac{w_k}{\|w_k\|}$  \\
   \> \textbf{for} $r=k:q$, \\
   \>    \> $Y_r^{(k+1)}=(I-2\,v_k v_k^{\top})Y_r^{(k)}$ \\
   \>\textbf{end}\\
   \textbf{end}
\end{tabbing}
where the following conventions have been used: 
\begin{itemize}
	\item $\|\cdot\|$ is the Euclidean norm.
	\item $Y^{(k+1)}=Q_kY^{(k)}$, $Y_r^{(k)}$ is column $r$ of $Y^{(k)}$.
	\item The projector $\Pi_k$ puts zeros in the first $k-1$ components and leaves the rest of the components unchanged when applied to a vector in $\mathbb{R}^m$.
	\item $\mathrm{e}_k$ is the $k$th canonical unit vector in $\mathbb{R}^m$.
\end{itemize}

The vectors $v_k$ computed in the algorithm contain all information needed to reconstruct the factor $Q$, whereas $R:=Y^{(q+1)}$. For simplicity, and to avoid the loss of regularity in $Q(y)$ viewed as a matrix valued function of $y$, we have here ignored the sign convention which is usually applied in the definition of $w_k$ \cite{trefethen97nla}. The idea is now to differentiate the variables in the algorithm with respect to $y$, writing for any object, say $X(y)$, its derivative as
$DX=DX(y,\zeta)=\left.\frac{\mathrm{d}}{\mathrm{d}\varepsilon}\right|_{\varepsilon=0}X(y+\epsilon\zeta)$ for any $y,\zeta\in\mathbb{R}^m$. The dependence on $y,\zeta$ will usually be suppressed in the notation when no confusion is at risk. Notice that $D$ commutes with $\Pi_k$ for any $k$.
The following recursion formulae are easily derived
\begin{align*}
Dw_k& = \Pi_kDY_k^{(k)}-\frac{(\Pi_kY_k^{(k)})^{\top} \Pi_kDY_k^{(k)}}{\|\Pi_kY_k^{(k)}\|}\mathrm{e}_k,\\
Dv_k &=\left(Dw_k - \frac{w_k^{\top}Dw_k}{\|w_k\|^2}\,w_k\right)\|w_k\|^{-1},\\
DY_r^{(k+1)}&=DY_r^{(k)}-2\left(
v_k^{\top}Y_r^{(k)}Dv_k + Dv_k^{\top}Y_r^{(k)}v_k + v_k^{\top}DY_r^{(k)}v_k
\right).
\end{align*}

The initial $DY^{(1)}=DY^{(1)}(y,\zeta)$ should be computed by differentiating the $q$ discrete gradients
$\dg H_1(p,y),\ldots,\dg H_q(p,y)$ with respect to $y$. When the discrete gradients are given by the AVF formula, we may derive the following expressions for the $r$th column of $DY^{(1)}$
$$
     DY_r^{(1)} (y,\zeta) =D\dg H_r(y,\zeta) = \left(\int_0^1 \xi\nabla^2 H_r(\xi y + (1-\xi)p)\mathrm{d}\xi\right)\cdot \zeta
$$
 where $\nabla^2 H_r$ is the Hessian of the integral $H_r$. The expression for the coordinate increment cases are given in the appendix. 

In the present application we only make use  of the $Q$-part of the QR-decomposition, thus we need only store $v_1,\ldots,v_q$ and
$Dv_1,\ldots,Dv_q$ for subsequent use. We may summarize the extended algorithm for computing these quantities as follows, using a 
Matlab inspired indexing notation where the submatrix $Y_{a:b,c:d}$ of $Y$ means
\begin{equation*}
	Y_{a:b,c:d}=
	\begin{pmatrix}
		Y_{a,c}&\cdots&Y_{a,d}\\
		\vdots&\ddots&\vdots\\
		Y_{b,c}&\cdots&Y_{b,d}
	\end{pmatrix}.
\end{equation*}

\begin{tabbing}
xxx\=xxxxxxxxxxxxx\=xxxxxxxxxxxxxxxxxxxxxxxxxxxxxxxxxxxxxxxxxxxxxxxxxxxxxxxxxxxx\kill
Given $Y$ and $DY$ as $m\times q$-matrices \\
\textbf{for} $k=1:q$  \\
\>$w=Y_{k:m,k}-\|Y_{k:m,k}\|\mathrm{e}_1$ \\[1mm]
\>$\displaystyle{Dw=DY_{k:m,k}-\frac{Y_{k:m,k}^{\top}DY_{k:m,k}}{\|Y_{k:m,k}\|}\mathrm{e}_1}$\\
\> $\tilde{v}_k=w/\|w\|$ \\
\> $\displaystyle{D\tilde{v}_k=(Dw-\frac{w^{\top}Dw}{\|w\|^2}\,w)\|w\|^{-1}}$\\[1mm]
\>$DY_{k:m,k+1:q}=DY_{k:m,k+1:q}-2\tilde{v}_k\tilde{v}_k^{\top}DY_{k:m,k+1:q}-
2D\tilde{v}_k\tilde{v}_k^{\top}Y_{k:m,k+1:q}$ \\[1mm]
\>\> $-2\tilde{v}_k\,D\tilde{v}_k^{\top}Y_{k:m,k+1:q}$\\[1mm]
\> $Y_{k:m,k+1:q}=Y_{k:m,k+1:q}-2\tilde{v}_k \tilde{v}_k^{\top} Y_{k:m,k+1:q}$ \\[1mm]
\textbf{end}
\end{tabbing}
The first $k-1$ entries of the vectors $v_{k}$, $Dv_{k}$ are zeros, and the remaining $m-k+1$ entries are
contained in $\tilde{v}_k$, $D\tilde{v}_k$ on exit. The complexity of this algorithm is $\mathcal{O}(mq^2+q^3)$.

 One should also note that we always multiply $Q$ ($DQ$ resp) by vectors $\bar{\eta}\in\mathbb{R}^m$ whose first $q$ columns are zero, this may be taken advantage of in the implementation. The procedure for computing $Q\bar{\eta}$ by means of $v_1,\ldots,v_k$ is described in
\cite[Algorithm 10.3]{trefethen97nla}, the cost is $\mathcal{O}(mq)$. We may extend this algorithm so that it computes also $DQ\,\bar{\omega}$
given $Dv_1,\ldots,Dv_q$. Defining the $m\times m$-matrix $P_k=Q_k\cdots Q_q,\ k=1,\ldots,q$, we get the downwards recursion
\begin{equation*}
P_{k-1} = Q_{k-1}P_k,\qquad P_q=Q_q,\quad P_1=Q=Q_1\cdots Q_q,
\end{equation*}
and differentiation yields
\begin{multline*}
DP_{k-1}=DQ_{k-1} P_k + Q_{k-1}DP_{k}\\=-2(Dv_{k-1}\,v_{k-1}^{\top}+v_{k-1}\,Dv_{k-1}^{\top}) P_{k}+ 
  (I-2v_{k-1}v_{k-1}^{\top})DP_{k}
\end{multline*}
The following algorithm results for computing $\phi=DQ\,\bar{\omega}$
\begin{tabbing}
xxx\=xxx\=xxxxxxxxxxxxxxxxxxxxxxxxxxxxxxxxxxxxxxxxxxxxxxx\kill
$\psi=\bar{\omega},\ \phi=0$\\
\textbf{for} $k=q:-1:1$ \\
\> $\phi=\phi-2(v_k^{\top}\psi\, Dv_{k}+Dv_k^{\top}\psi\, v_k+v_k^{\top}\phi\,v_k)$\\
 \> $\psi=\psi-2v_k^{\top} \psi\, v_k$\\
 \textbf{end}
\end{tabbing}
The complexity of this algorithm is $\mathcal{O}(mq)$.

\section{Numerical integration of the Kepler problem}\label{se:numex}
	The Kepler two-body problem describes the motion of two bodies which attract each other. By placing the first body in the origin, the position $(y_1,y_2)$ and the velocity $(y_3,y_4)$ of the other body are given by the following four-dimensional ODE
	\begin{align}
	\begin{split}
		\dot{y}_1&=y_3,\\
		\dot{y}_2&=y_4,\\
		\dot{y}_3&=-\frac{y_1}{(y_1^2+y_2^2)^{3/2}},\\			
		\dot{y}_4&=-\frac{y_2}{(y_1^2+y_2^2)^{3/2}}.
	\end{split}
	\label{eq:kepler}
	\end{align}
	This system preserves the Hamiltonian
	\begin{align*}
		H_1(y)&=\frac12\left(y_3^2+y_4^2\right)-\frac1{\sqrt{y_1^2+y_2^2}},
		\intertext{the angular momentum}
		H_2(y)&=y_1y_4-y_2y_3,
		\intertext{and the Runge-Lenz-Pauli vector}
		H_3(y)&=y_2y_3^2-y_1y_3y_4-\frac{y_2}{\sqrt{y_1^2+y_2^2}},\\
		H_4(y)&=y_1y_4^2-y_2y_3y_4-\frac{y_1}{\sqrt{y_1^2+y_2^2}}.
	\end{align*}

	Since $q=m=4$ any subset of three out of the four invariants is dependent. We want to compare schemes that preserve none, one, two, and all of the invariants above. 
	We use the projection method \eqref{scm:A} with the standard fourth order explicit Runge-Kutta method as the underlying scheme. The discrete gradients are calculated using \eqref{eq:symitohabe}. The schemes that preserve one of $H_1,H_3$ are denoted as RK4Proj1 and RK4Proj3, respectively. 
	The scheme that preserves both $H_1$ and $H_3$ is called RK4Proj13. The original Runge-Kutta scheme preserves neither and is denoted RK4.
	 The RK123Proj scheme is omitted from the plot since it produces exactly the ellipsis of the Kepler problem.
	 
	The resulting plots from these methods in Figure \ref{fi:keplersolution} are arranged according to the table
	\begin{table}[!ht]
		\centering
		\begin{tabular}{|c|c|}
			\hline
			RK4&RK4Proj1\\
			\hline
			RK4Proj3&RK4Proj13\\
			\hline
		\end{tabular}
		\caption{The location of the schemes in the plots of Figure \ref{fi:keplersolution}.}
		\label{ta:schms}
	\end{table}
	
	The initial values are taken from section I.2.3 of \cite{hairer06gni},
	\begin{equation*}
		y_1^0=1-e,\quad y_2^0=0,\quad y_3^0=0,\quad y_4^0=\sqrt{\frac{1+e}{1-e}},
	\end{equation*}
	where the eccentricity is $e=0.6$ and the exact solution has period $2\pi$. The time step is $h=0.2$ and we integrate for 50000 steps.
	
	Figure \ref{fi:keplersolution} shows the numerical solutions. RK4 spirals inwards until it eventually blows up. RK4Proj1 has a counterclockwise precession effect. The Runge-Lenz-Pauli vector has to do with the orientation of the ellipse and RK4Proj3 does therefore not exhibit this effect, it will however converge to a smaller circle around the origin. The solution of RK4Proj13 shows an improvement compared to RK4Proj1 and RK4Proj3. 
	
This example illustrates that there are cases where the preservation of one or more invariants are important to get a numerical solution with good long term properties. Not surprisingly, one observes a gradual improvement in the quality of the solution as the number of preserved first integrals increases.
 The extra computational effort needed to preserve multiple integrals compared to one is almost negligible, in this example the computation took less than 10\% longer.

\begin{figure}
		\centering
		\includegraphics[width=\textwidth]{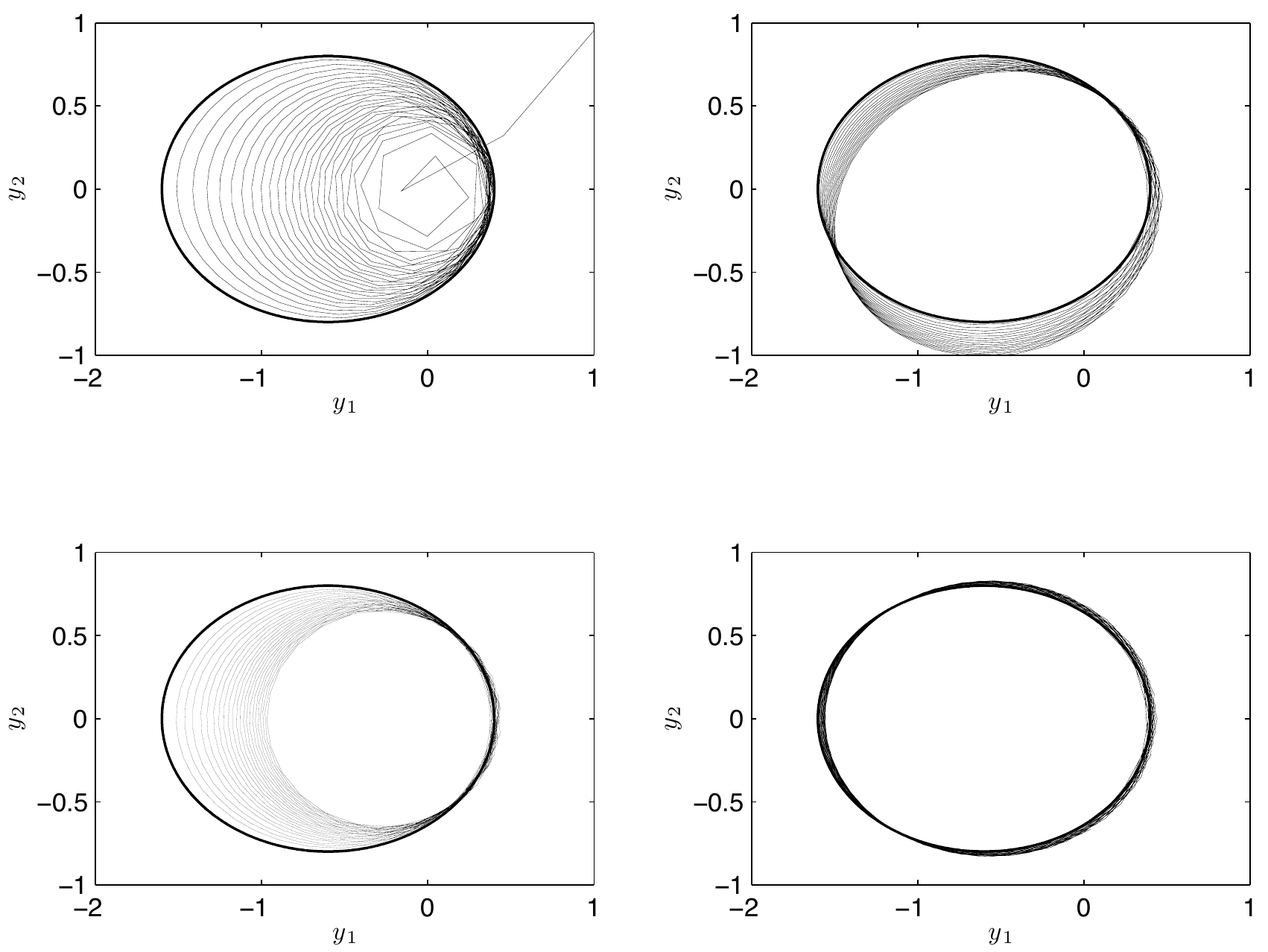}
		\caption{The numerical solution (thin line) of the Kepler problem \eqref{eq:kepler} using the schemes of Table \ref{ta:schms} with $h=0.2$. The first 500 steps are shown. The exact solution (thick line) is an ellipse with eccentricity $e=0.6$. }
		\label{fi:keplersolution}
	\end{figure}
	
 In Figure \ref{fi:keplerorder_2358} we plot the global error of four schemes (RK2Proj123, RK4Proj123, RK5Proj123, and RK7Proj123) that preserve the four invariants. The underlying schemes are four RK-schemes of order 2, 4, 5, and 7. We see that the schemes attain the order of the underlying scheme, which is what we proved in Theorem \ref{thm:A}. 	
 	
	
	\begin{figure}
		\centering
		\includegraphics[width=\textwidth]{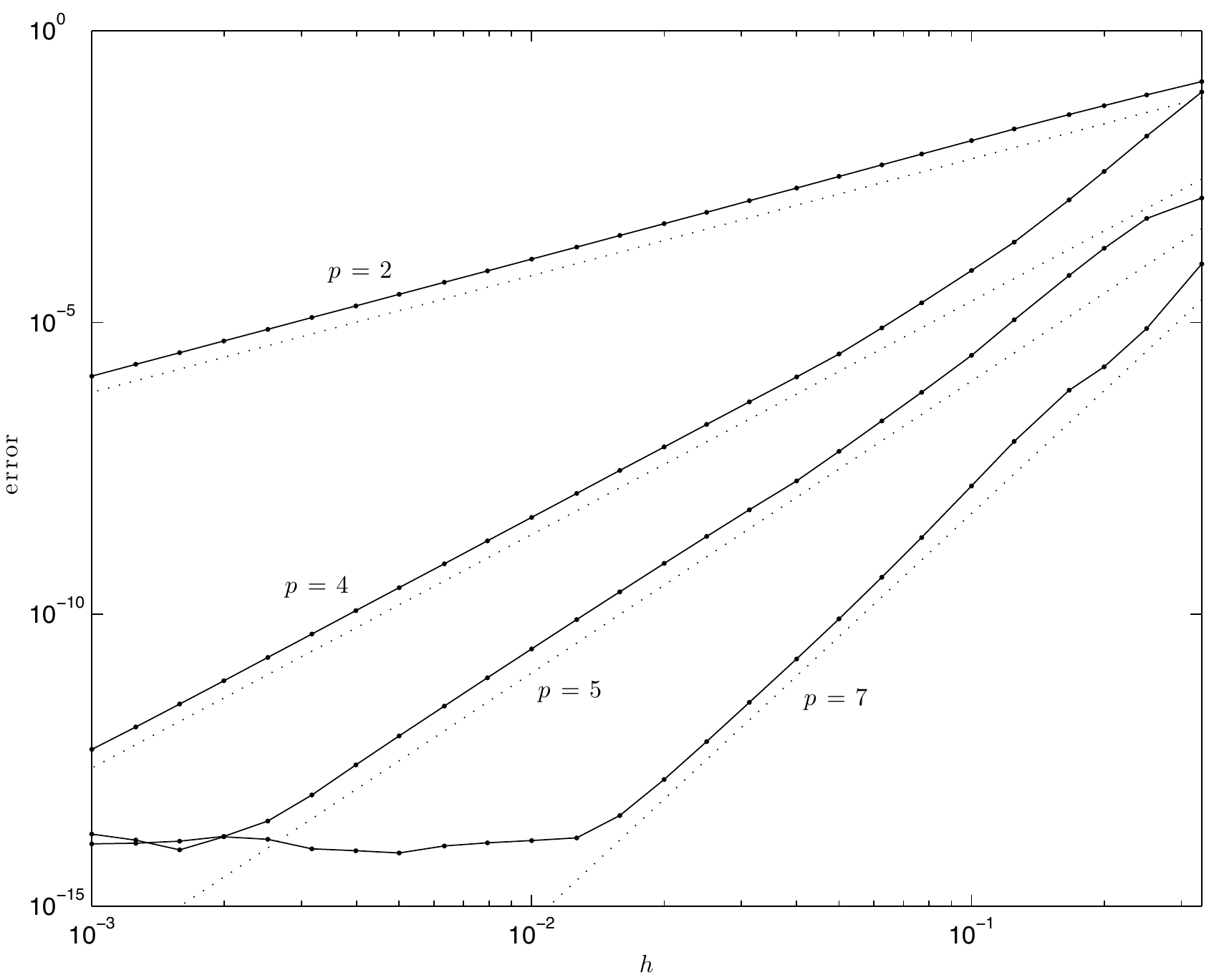}
		\caption{The global error of the four schemes RK2Proj123, RK4Proj123, RK5Proj123, and RK7Proj123. The dotted lines are reference lines of exact order.}
		\label{fi:keplerorder_2358}
	\end{figure}
Figure  \ref{fi:comparisonplot_mid} shows a comparison between a standard orthogonal projection method as defined at the end of Section \ref{se:proj} and the projection method \eqref{scm:A}. Both preserve $H_1$ and $H_2$ simultaneously and use the implicit midpoint method as the underlying scheme. Figure \ref{e06h01} shows that our projection method is more accurate for $e=0.6$ while Figure \ref{e07h005} shows that the standard projection method is more accurate for $e=0.7$. 
In the standard projection method one has to compute the distance to the underlying manifold, usually denoted by the Lagrange multiplier $\lambda$, however for our projection scheme this is already (implicitly) known. The resulting nonlinear system will have dimension $m+q$ (see section IV.4 in \cite{hairer06gni}) compared to $m$ for our proposed method. Our implementation uses the same nonlinear solver (Matlab's fsolve) for both methods. In Figure \ref{e07h005h0075} we have adjusted $h$ such that both schemes have the same computation time, in which case one sees that the new method is slightly better than the standard projection method even for $e=0.7$.
\begin{figure}
	\centering
	\subfloat[$e=0.6$ and $h=0.1$.]{\label{e06h01}\includegraphics[width=\textwidth]{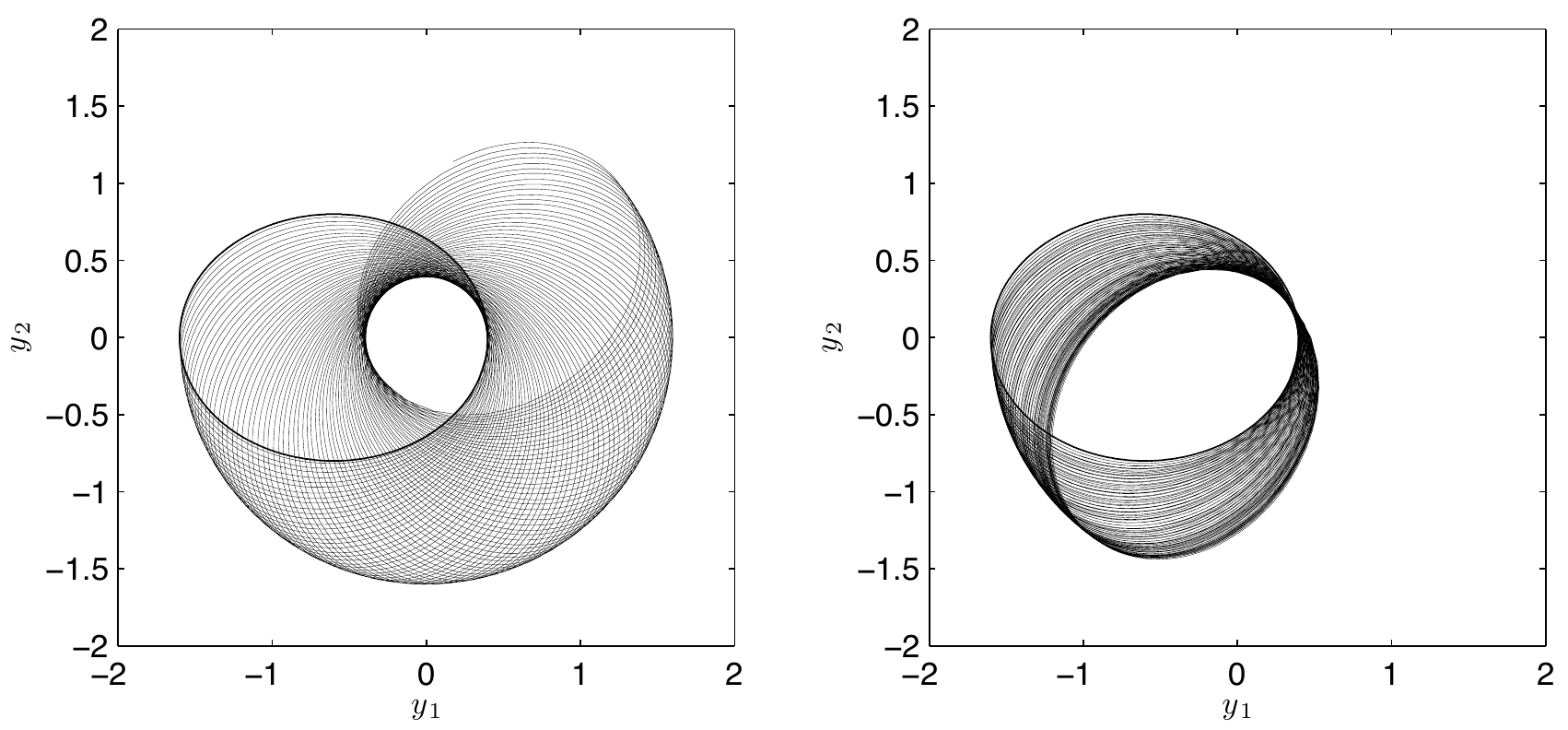}}\\
	\subfloat[$e=0.7$ and $h=0.05$.]{\label{e07h005}\includegraphics[width=\textwidth]{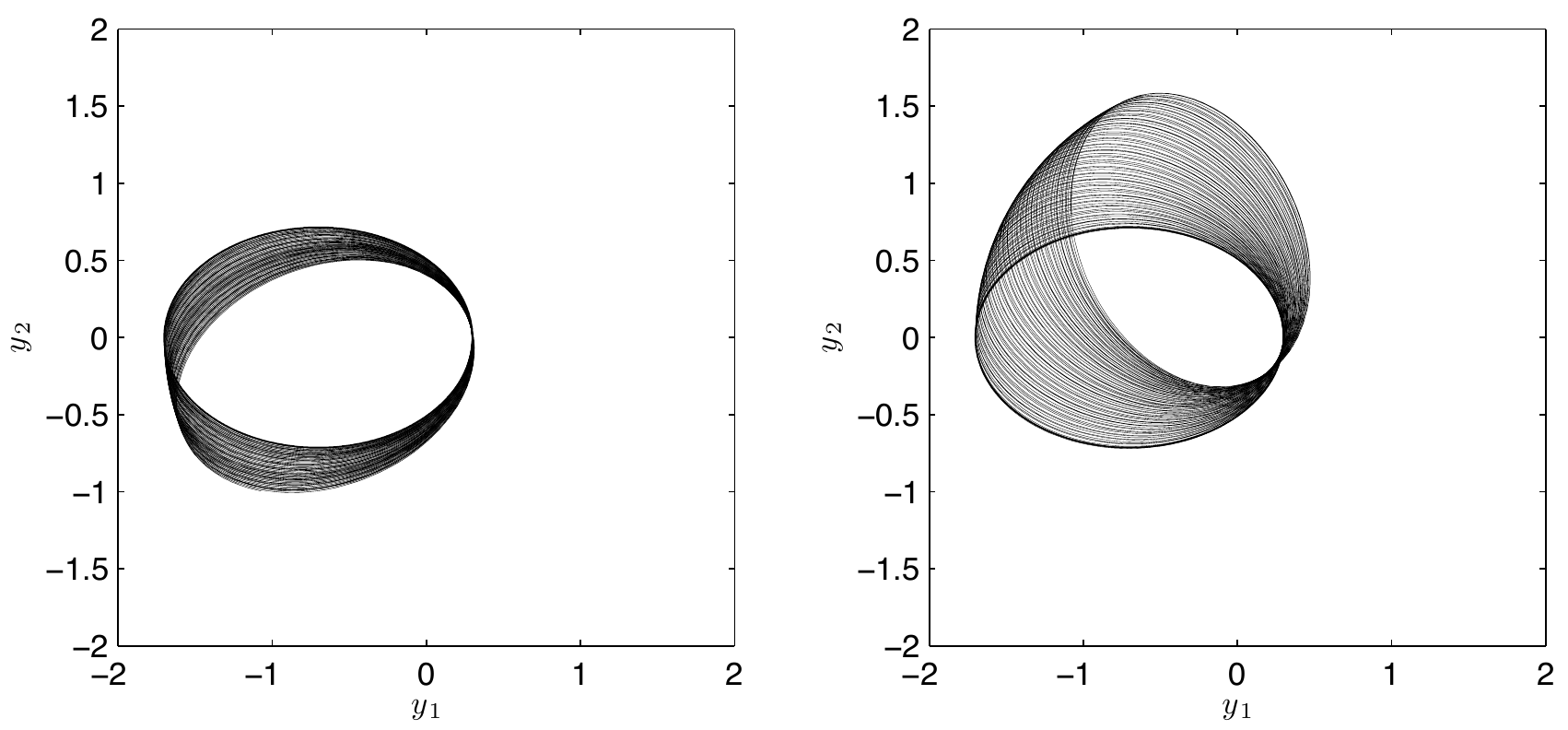}}\\
	\subfloat[$e=0.7$ and $h=0.075$ (left) $h=0.05$ (right).]{\label{e07h005h0075}\includegraphics[width=\textwidth]{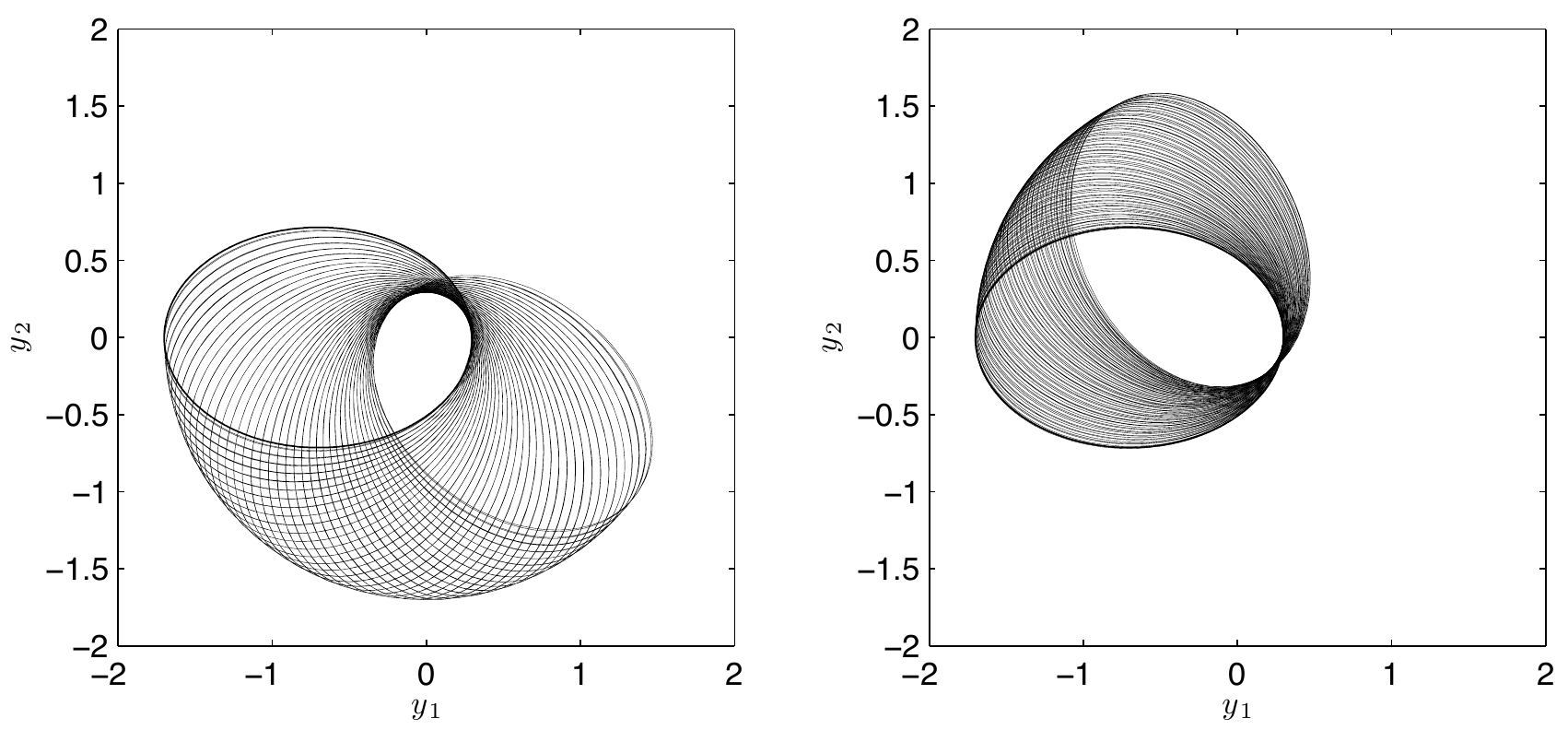}}	
	\caption{The numerical solution (thin line) of the Kepler problem \eqref{eq:kepler} using the midpoint method with standard orthogonal projection (left) and \eqref{scm:A} with the midpoint method as the underlying scheme. The exact solution (thick line) is an ellipse with eccentricity $e$. Both schemes preserves $H_1$ and $H_2$ exactly. In the last row (c) the step size $h$ is adjusted such that both schemes have the same computation time. }
	\label{fi:comparisonplot_mid}
\end{figure}

\paragraph{Conclusion and further work.}  We have presented a new methodology for preserving multiple first integrals in systems of ordinary differential equations, using discrete gradients as the underlying tool. By using the notion of a discrete tangent space, two methodologies for designing numerical schemes are easily derived, projection and local coordinates. The resulting algorithms are relatively inexpensive compared to well-known algorithms preserving precisely one first integral.  
There are of course several other well known methods for preserving multiple invariants, but we believe that the new methods which are based on discrete gradients rather than exact ones may be attractive for certain problems. Symmetric schemes are easily constructed, no Lagrange multipliers are required, and the schemes can be seen as natural generalizations of the popular discrete gradient methods.
Although the present paper considers only systems of ordinary differential equations, the approach taken may be easily adapted to partial differential equations to be considered in future work.
	
\paragraph{Acknowledgements}
We are grateful to the anonymous referees for helpful comments and references.
	
\appendix
\section{Derivative of discrete gradients of the Coordinate Increment type}
We here give the expression for the derivative map of the discrete gradients defined in terms of the coordinate increment method of 
Itoh and Abe \cite{itoh88hcd}.
We write \eqref{eq:itohabe} in compact notation as
\begin{equation*}
 (\dgci H(v,u))_i = \frac{H(u|_iv)-H(u|_{i-1}v)}{u_i-v_i},\quad 1\leq i\leq m
\end{equation*}
The jacobian of this map with respect to its second argument is then simply the lower-triangular matrix with elements
\begin{equation*}
(D\dgci H(v,u))_{ij} =\left\{ 
\begin{array}{ll}
  \displaystyle{  \frac{\frac{\partial}{\partial u_j}(H(u |_i v)-H(u |_{i-1} v))}{u_i-v_i}},& j<i\\[3mm]
  \displaystyle{ \frac{ \frac{\partial H}{\partial u_i}(u|_i v)}{u_i-v_i}-
   \frac{H(u |_i v) - H(u |_{i-1}v)}{(u_i-v_i)^2}},& j=i \\[3mm]
  0& j>i
\end{array}
\right.
\end{equation*}
Similarly, we can compute the Jacobian of the symmetrised version as follows
\begin{equation*}
(D\dgsci H(v,u))_{ij} =
\begin{cases}
	\frac{1}{2}  \frac{\frac{\partial}{\partial u_j}(H(u |_i v)-H(u |_{i-1} v))}{u_i-v_i},&j<i\\
	\frac12\frac{ \frac{\partial}{\partial u_i}(H(u|_i v)+H(v|_{i-1}u))}{u_i-v_i}+
   	\frac12\frac{H(v|_i u)+H(u |_{i-1}v)-H(v|_{i-1}u)-H(u |_i v)}{(u_i-v_i)^2},& j=i\\
    	\frac{1}{2} \displaystyle{  \frac{\frac{\partial}{\partial u_j}(H(v |_{i-1} u)-H(v |_{i} u))}{u_i-v_i}},& j>i.
\end{cases}
\end{equation*}

 \bibliographystyle{plain}
 \bibliography{dgm}

\def\cprime{$'$}
\begin{thebibliography}{10}

\bibitem{celledoni02aco}
E.~Celledoni and B.~Owren.
\newblock A class of intrinsic schemes for orthogonal integration.
\newblock {\em SIAM J. Numer. Anal.}, 40(6):2069--2084 (electronic) (2003),
  2002.

\bibitem{cieslinski07}
J.~L. Cie{\'s}li{\'n}ski.
\newblock An orbit-preserving discretization of the classical {K}epler problem.
\newblock {\em Phys. Lett. A}, 370(1):8--12, 2007.

\bibitem{cieslinski10}
J.~L. Cie{\'s}li{\'n}ski.
\newblock Comment on `{C}onservative discretizations of the {K}epler motion'.
\newblock {\em J. Phys. A}, 43(22):228001, 4, 2010.

\bibitem{courant28udp}
R.~Courant, K.~Friedrichs, and H.~Lewy.
\newblock \"{U}ber die partiellen {D}ifferenzengleichungen der mathematischen
  {P}hysik.
\newblock {\em Math. Ann.}, 100(1):32--74, 1928.

\bibitem{gonzalez96tia}
O.~Gonzalez.
\newblock Time integration and discrete {H}amiltonian systems.
\newblock {\em J. Nonlinear Sci.}, 6(5):449--467, 1996.

\bibitem{hairer06gni}
E.~Hairer, C.~Lubich, and G.~Wanner.
\newblock {\em Geometric numerical integration}, volume~31 of {\em Springer
  Series in Computational Mathematics}.
\newblock Springer-Verlag, Berlin, second edition, 2006.
\newblock Structure-preserving algorithms for ordinary differential equations.

\bibitem{itoh88hcd}
T.~Itoh and K.~Abe.
\newblock Hamiltonian-conserving discrete canonical equations based on
  variational difference quotients.
\newblock {\em J. Comput. Phys.}, 76(1):85--102, 1988.

\bibitem{kozlov07cdo}
R.~Kozlov.
\newblock Conservative discretizations of the {K}epler motion.
\newblock {\em J. Phys. A}, 40(17):4529--4539, 2007.

\bibitem{labudde74dma}
R.~A. LaBudde and D.~Greenspan.
\newblock Discrete mechanics---a general treatment.
\newblock {\em J. Computational Phys.}, 15:134--167, 1974.

\bibitem{labudde76eai}
R.~A. LaBudde and D.~Greenspan.
\newblock Energy and momentum conserving methods of arbitrary order of the
  numerical integration of equations of motion. {I}. {M}otion of a single
  particle.
\newblock {\em Numer. Math.}, 25(4):323--346, 1975/76.

\bibitem{labudde76eat}
R.~A. LaBudde and D.~Greenspan.
\newblock Energy and momentum conserving methods of arbitrary order for the
  numerical integration of equations of motion. {II}. {M}otion of a system of
  particles.
\newblock {\em Numer. Math.}, 26(1):1--16, 1976.

\bibitem{mclachlan99giu}
R.~I. McLachlan, G.~R.~W. Quispel, and N.~Robidoux.
\newblock Geometric integration using discrete gradients.
\newblock {\em R. Soc. Lond. Philos. Trans. Ser. A Math. Phys. Eng. Sci.},
  357(1754):1021--1045, 1999.

\bibitem{minesaki04anc}
Y.~Minesaki and Y.~Nakamura.
\newblock A new conservative numerical integration algorithm for the
  three-dimensional {K}epler motion based on the {K}ustaanheimo-{S}tiefel
  regularization theory.
\newblock {\em Phys. Lett. A}, 324(4):282--292, 2004.

\bibitem{minesaki06nni}
Y.~Minesaki and Y.~Nakamura.
\newblock New numerical integrator for the {S}t\"ackel system conserving the
  same number of constants of motion as the degree of freedom.
\newblock {\em J. Phys. A}, 39(30):9453--9476, 2006.

\bibitem{MR1142054}
F.~A. Potra and W.~C. Rheinboldt.
\newblock On the numerical solution of {E}uler-{L}agrange equations.
\newblock {\em Mech. Structures Mach.}, 19(1):1--18, 1991.

\bibitem{MR1142057}
F.~A. Potra and J.~Yen.
\newblock Implicit numerical integration for {E}uler-{L}agrange equations via
  tangent space parametrization.
\newblock {\em Mech. Structures Mach.}, 19(1):77--98, 1991.

\bibitem{MR2451073}
G.~R.~W. Quispel and D.~I. McLaren.
\newblock A new class of energy-preserving numerical integration methods.
\newblock {\em J. Phys. A}, 41(4):045206, 7, 2008.

\bibitem{quispel99son}
G.R.W. Quispel and H.~Capel.
\newblock Solving {ODE}'s numerically while preserving all first integrals,
  1999.
\newblock Preprint
  (http://www.latrobe.edu.au/mathstats/\\staff/quispel/quispel/Publ55\_Solving\%20ODE's\%20numerically.pdf).

\bibitem{MR1187632}
J.~C. Simo, N.~Tarnow, and K.~K. Wong.
\newblock Exact energy-momentum conserving algorithms and symplectic schemes
  for nonlinear dynamics.
\newblock {\em Comput. Methods Appl. Mech. Engrg.}, 100(1):63--116, 1992.

\bibitem{stiefel71lar}
E.~L. Stiefel and G.~Scheifele.
\newblock {\em Linear and regular celestial mechanics. {P}erturbed two-body
  motion, numerical methods, canonical theory}.
\newblock Springer-Verlag, New York, 1971.
\newblock Die Grundlehren der mathematischen Wissenschaften, Band 174.

\bibitem{trefethen97nla}
L.~N. Trefethen and D.~Bau.
\newblock {\em Numerical linear algebra}.
\newblock Society for Industrial and Applied Mathematics (SIAM), Philadelphia,
  PA, 1997.

\end{thebibliography}

\end{document}